\documentclass[11pt]{amsart}
\usepackage[top=90pt,bottom=90pt,left=90pt,right=90pt]{geometry}
\usepackage[utf8]{inputenc}
\usepackage[T1]{fontenc}
\usepackage[english]{babel}
\usepackage{enumerate}
\usepackage[dvips]{graphicx}
\usepackage[size=footnotesize]{caption}
\usepackage[size=footnotesize]{subcaption}
\usepackage{subcaption}
\usepackage{fullpage}
\usepackage{bbm,amsmath,amssymb,amsfonts,graphicx,xcolor,mathtools, verbatim,color,epsf}
\usepackage[bookmarksopen=false,pdftex=true,breaklinks=true,%
      hyperindex=true,pdfstartview=FitH,colorlinks=true,%
      pdfpagelabels=true,colorlinks=true,linkcolor=black,%
      citecolor=black,urlcolor=black,hypertexnames=false%
      ]%
   {hyperref}

\newcommand{\R}{\mathbb{R}}
\renewcommand{\S}{\mathcal{S}}
\newcommand{\C}{\mathbb{C}}
\newcommand{\X}{\underline{X}}

\newcommand{\G}{\mathbb{G}}

\newcommand{\HH}{\mathbb{H}}
\newcommand{\HW}{\mathcal{H}\hspace{-0.55mm}\mathcal{W}}
\renewcommand{\leq}{\leqslant}
\usepackage{faktor}
\usepackage{MnSymbol} 
\renewcommand{\Re}{\operatorname{Re}}
\renewcommand{\Im}{\operatorname{Im}}
\usepackage{etoolbox}
\makeatletter
\patchcmd{\@setaddresses}{\indent}{\noindent}{}{}
\patchcmd{\@setaddresses}{\indent}{\noindent}{}{}
\patchcmd{\@setaddresses}{\indent}{\noindent}{}{}
\makeatother

\theoremstyle{definition}
\newtheorem{theorem}{Theorem}[section]
\newtheorem{proposition}[theorem]{Proposition}
\newtheorem{corollary}[theorem]{Corollary}
\newtheorem{lemma}[theorem]{Lemma}

\newtheorem{question}[theorem]{Question}
\newtheorem{definition}[theorem]{Definition}
\newtheorem{example}[theorem]{Example}
\newtheorem{remark}[theorem]{Remark}

\DeclareMathOperator{\conv}{conv}

\DeclareFontFamily{U}{mathx}{\hyphenchar\font45}
\DeclareFontShape{U}{mathx}{m}{n}{
      <5> <6> <7> <8> <9> <10>
      <10.95> <12> <14.4> <17.28> <20.74> <24.88>
      mathx10
      }{}
\DeclareSymbolFont{mathx}{U}{mathx}{m}{n}
\DeclareFontSubstitution{U}{mathx}{m}{n}
\DeclareMathAccent{\widecheck}{0}{mathx}{"71}

\title{Slices of Stable Polynomials and Connections to the Grace-Walsh-Szeg\H{o} Theorem}

\author{Sebastian Debus}
\address{Sebastian Debus, Technische Universität Chemnitz, Fakultät für Mathematik, 09107 Chemnitz, Germany}
\email{sebastian.debus@mathematik.tu-chemnitz.de}

\author{Cordian Riener}
\address{Cordian Riener, Department of Mathematics and Statistics, UiT - The Arctic University of Norway, 9037 Troms\o, Norway}
\email{cordian.riener@uit.no}

\author{Robin Schabert}
\address{Robin Schabert, Fachbereich Mathematik und Statistik, Universität Konstanz, 78457 Konstanz, Germany}
\email{robin.schabert@uni-konstanz.de}

\date{\scriptsize{\today}}

\thanks{\scriptsize{Cordian Riener and Robin Schabert were supported by Tromsø Research Foundation under the grant agreement 17matteCR (SymRAG)}}

\begin{document}
\setlength{\parindent}{0pt}

\begin{abstract} 
Univariate polynomials are called stable with respect to a domain $D$ if all of their roots lie in $D$. 
We study linear slices of the space of stable univariate polynomials with respect to a half-plane.
We show that a linear slice always contains a stable polynomial with only a few distinct roots.
Subsequently, we apply these results to symmetric polynomials and varieties. 
We show that for varieties defined by few multiaffine symmetric polynomials, the existence of a point in $D^n$ with few distinct coordinates is necessary and sufficient for the intersection with $D^n$ to be non-empty.
This is at the same time a generalization of the so-called degree principle to stable polynomials and a result similar to Grace-Walsh-Szeg\H{o}'s coincidence theorem.
\end{abstract}

\maketitle

\section*{Introduction}

The study of univariate polynomials whose roots lie in prescribed regions of the complex plane is a classical and central topic in mathematics, with deep connections to algebra, analysis, combinatorics, and control theory. A particularly important case is that of \emph{hyperbolic polynomials}, i.e., real polynomials whose roots are all real. More generally, one considers \emph{$D$-stable polynomials} - those polynomials whose roots are contained in a designated domain, often a \emph{circular region} $D \subset \C$, meaning a region bounded by a circle or a line, possibly open or closed.

A case of central interest in this work is when $D$ is a half-plane. For instance, a real polynomial is hyperbolic precisely when it is stable with respect to the closed upper half-plane, due to the symmetry of complex conjugate roots. Classical and widely studied examples of such stable polynomials include \emph{Hurwitz stable polynomials}, whose roots lie in the open left half-plane and are fundamental in control theory.

Beyond their classical importance, stable polynomials also feature prominently in modern developments in combinatorics and theoretical computer science, for example, in the theory of real stable polynomials and their applications to negative dependence and log-concavity (see, e.g., \cite{branden2007polynomials,fisk2008polynomials}).

A powerful perspective arises when viewing these polynomials via their roots, encoded through \emph{Vieta's formula}. A monic univariate polynomial of degree $n$ with roots $x_1,\dots,x_n$ can be written as
\[
f_z(T) = T^n - z_1 T^{n-1} + z_2 T^{n-2} - \ldots + (-1)^n z_n,
\]
where $z_i = e_i(x)$ is the $i$-th elementary symmetric polynomial in the roots. This identifies the space of monic polynomials with $\C^n$, via the \emph{Vieta map} from $\C^n$ to symmetric functions of roots. In this setting, the set of hyperbolic polynomials corresponds to the image of $\R^n$ under the Vieta map, while $D$-stable polynomials arise as the image of $D^n$.

Within this framework, we investigate \emph{slices} of the set of $D$-stable polynomials, i.e., subsets defined by fixing linear combinations of coefficients. For hyperbolic polynomials, such slices, often called \emph{hyperbolic slices} or \emph{Vandermonde varieties}, have been extensively studied starting from Arnold’s seminal work \cite{Arnold,Givental,Kostov,Meguerditchian} up to recent contributions \cite{Schabert-Lien,riener2012degree}. These slices reveal rich geometric and combinatorial structures.

These \emph{stable slices} arise as natural geometric sections of the space of $D$-stable polynomials and offer a powerful lens through which to study the interplay between algebraic constraints and root configurations. By fixing linear relations among coefficients, one obtains linear subspaces intersecting the semialgebraic set of stable polynomials, giving rise to highly structured and tractable subsets. The geometry of these slices reflects a rich \emph{stratification by root multiplicities and symmetries}, and our results reveal a striking \emph{sparsity phenomenon}: the number of non-real and distinct real roots in local extreme points is tightly controlled by the slice's codimension. This geometric perspective generalizes classical constructions such as \emph{Vandermonde varieties} and \emph{hyperbolic amoebas}, and connects naturally to applications in real algebraic geometry, \emph{stability analysis in control theory}, and \emph{optimization over structured polynomial spaces}, such as those arising in signal processing and spectral graph theory. By identifying low-complexity representatives within slices, we also facilitate reductions of stability verification problems to simpler, lower-dimensional cases.

In this paper, we generalize these ideas from hyperbolic to \emph{upper half-plane stable} polynomials. We study linear slices of the set of stable polynomials defined by $k$ linear constraints on coefficients, and show in Theorem~\ref{maintheorem} that the local extreme points of such stable slices have at most $k$ non-real roots and at most $2k$ distinct real roots. This extends classical sparsity results into the complex setting and provides new tools for analyzing the structure of stable polynomials under linear constraints.

A key motivation stems from connections to the \emph{Grace-Walsh-Szeg\H{o} coincidence theorem}, a classical result stating that for symmetric multiaffine polynomials evaluated over a circular region $D$, any tuple in $D^n$ can be replaced by a constant tuple $(\zeta, \dots, \zeta)$ without changing the evaluation. While powerful, this result relies on strong symmetry assumptions. Recent work by Bränd\'en and Wagner \cite{Bränden-Wagner} shows such results generally fail for polynomials invariant under proper subgroups of $S_n$.

We show that for certain multivariate polynomials built from a small number of symmetric multiaffine building blocks, a similar ``coincidence'' result \emph{does} hold when $D$ is a half-plane. Specifically, we prove (Theorem~\ref{thm:grace}, Proposition~\ref{prop:main2}) that for any point in $D^n$, the evaluation of such functions agrees with that at a point with few distinct coordinates. In this way, stable slices provide a natural setting to generalize classical theorems under weaker assumptions.

We also introduce and prove a \emph{double-degree principle} for varieties of symmetric polynomials (Theorem~\ref{degree-principle}) and a \emph{half-degree principle} specific to upper half-plane stability (Theorem~\ref{thm:half-degree}), extending classical ideas about root sparsity to this broader context.

\subsection*{Relation to previous work}
As stated before, the ideas of studying slices of hyperbolic polynomials go back to the work of Arnold and his school. This paper builds on the ideas and techniques developed in \cite{riener2012degree, Riener-Schabert}, where the authors studied hyperbolic slices and positivity conditions for symmetric multivariate polynomials. 
In those works, the focus was on hyperbolic polynomials, whose roots lie on the real axis, and the results relied heavily on the geometry of real-rooted polynomials. 

In contrast, the present paper studies stable slices, where the roots are constrained to lie within (closed) half-planes. 
The shift from hyperbolicity to stability introduces fundamentally different geometric behavior, particularly concerning the interaction between interior and boundary points of the domain. 
While some structural methods from \cite{riener2012degree, Riener-Schabert} can be adapted quite directly, the main results here --- such as the connection to the Grace--Walsh--Szeg\H{o} theorem (Theorem 2.3) and the double-degree principle (Theorem 2.7) do not follow directly from the earlier hyperbolic case.

\section{slices of $D$-stable polynomials} \label{sec:Stable slices}
Throughout the article we denote by $\C[T]$ and $\R[T]$ the rings of univariate complex and real polynomials and fix positive integers $k \leq n$. Throughout, we identify \( \mathbb{C}^n \) with the real vector space \( \mathbb{R}^{2n} \) via the mapping
\[
(x_1 + i y_1, \ldots, x_n + i y_n) \longmapsto (x_1, y_1, \ldots, x_n, y_n),
\]
and all notions of convexity, convex combinations, and extreme points are considered with respect to this real structure.

For a complex number $z$ we write $\Re(z)$ and $\Im(z)$ for its real and imaginary parts. Furthermore, we commonly identify the set of monic univariate polynomials with $\C^n$ via the bijection
\[(z_1,\dots,z_n)\longmapsto f_z(T):=T^n-z_1T^{n-1}+z_2T^{n-2}-\dots + (-1)^n z_n.\]
In this section we are interested in polynomials whose roots lie in a prescribed domain $D\subset\C$.

\begin{definition}[\(D\)-stable polynomials and slices]\label{def:slices}
Let \(D \subset \C\) be a subset. A polynomial \(f \in \C[T]\) (or $f\in\R[T])$  of degree \(n\) is called \emph{\(D\)-stable} (or \emph{real \(D\)-stable}) if all of its roots lie in \(D\). The set of all such monic polynomials is denoted by
\[
\mathcal{S}_D(\mathbb{F}) := \left\{ z \in \mathbb{F}^n ~:~ f_z(T) = T^n - z_1 T^{n-1} + \dots + (-1)^n z_n \text{ has all roots in } D \right\},  
\] where $\mathbb{F}$ is either $\C$ or $\R$.
For a surjective linear map \(L : \C^n \to \C^k\) (or $M : \R^n \to \R^k$) and \(a \in \C^k\) (or $b \in \R^k$), the sets
\begin{align*}
 \mathcal{S}_D(\C) \cap L^{-1}(a), & &   \mathcal{S}_D(\R) \cap M^{-1}(b)   
\end{align*}
are called a \emph{slice of \(D\)-(real) stable polynomials}, or simply a (real) \emph{\(D\)-stable slice}. 
\end{definition}
A $D$-stable slice is thus the set of monic (real) $D$-stable polynomials whose coefficients satisfy $k \leq n$ linearly independent relations.  
\begin{remark}[Common domains \(D\)]
Several well-known classes of polynomials arise by choosing specific domains \(D \subset \C\):
\begin{enumerate}
    \item \textbf{Hyperbolic polynomials:} \(D = \R\). Then \(\mathcal{S}_\R(\R) = \mathcal{S}_\R(\C)\) is the set of monic polynomials with only real roots.
    \item \textbf{Hurwitz polynomials:} If \(D := \{ z \in \C : \Re(z) \leq 0 \}\) then the real $D$-stable polynomials  whose roots lie in the interior of $D$ are called \emph{Hurwitz polynomials} and the real $D$-stable polynomials are called \emph{weakly Hurwitz}. 
    We write \(\HW\) instead of \(\S_{D}(\R)\) and call a stable slice $\HW \cap M^{-1}(b)$ a \emph{Hurwitz slice}.
    \item \textbf{Schur-stable polynomials:} \(D = \{ z \in \C : |z| < 1 \}\). These are polynomials whose roots lie in the open unit disk. Hurwitz and Schur-stable polynomials are used in control theory and complex analysis since they guarantee that all solutions of the associated systems of
differential and difference equation converge to $0$.
        \item \textbf{Upper half-plane stable polynomials:} For \(D = \HH_+ := \{ z \in \C : \Im(z) \geq 0 \}\) the closed upper half-plane we write \(\S\) instead of \(\S_{\HH_+}(\C)\).
    \item \textbf{General domains:} For applications, \(D\) can also be a strip, sector, cone, or polygonal region, depending on the stability conditions relevant to the problem.
\end{enumerate}
\end{remark}

In the setting of hyperbolic polynomials, it had been observed that the corresponding hyperbolic slices exhibit the following very interesting properties: Given a hyperbolic slice defined by $k$ linear conditions. Then every affine linear function achieves its minimum or maximum over the slice at a polynomial with at most $k$ distinct roots (\cite[Theorem 2.8]{Riener-Schabert}). Since the set of such polynomials inside a hyperbolic slice is finite, this implies that the convex hull of the slice is a polytope (see Figure \ref{fig:hyperbolic} for an example of such a slice).
This property is a consequence of the strong concavity of the discriminant variety around the hyperbolic polynomials, and besides its geometrical consequences, it also has algorithmic implications leading to algorithmic simplifications, as, for example, shown in \cite{basu2022vandermonde,riener2025deciding}. \\
The focus here is to investigate \(\HH\)-\emph{stability}, where \(\HH\) is a half-plane. More specifically, we are interested in intersections of the set of stable polynomials with subspaces of $\C^n$. As multiplication with units in $\C$ does not change the roots of a polynomial, we restrict to monic stable polynomials. Given the results on hyperbolic slices, it is natural to wonder, whether this property generalizes to more general $D$-stable slices. As can be seen, for example, in Figure \ref{fig:slice1}, these slices do no longer possess the same strong concavity of the boundary and their convex hulls are not spanned by finitely many points, in general. However, we show below that also in this setup the convex hull is spanned by polynomials with restricted root multiplicities.

\begin{example}\label{ex:stable-slice}
We consider $\S \cap L^{-1}(a)$, where
\[a:=(-23 i, -463, 8461 i) \quad \text{and} \quad L : {\C^4} \to {\C^3}, {(z_1,z_2,z_3,z_4)} \mapsto {(z_1,z_2,z_3)}\]
is the projection to the first $3$ coordinates. Then $\S \cap L^{-1}(a)$ is non-empty, since \[(-23 i, -463, 8461 i, 8020)\in \S \cap L^{-1}(a).\] The coefficient vector corresponds to a polynomial with roots $-20+i,i,20+i$ and $20i$. Furthermore, $\S \cap L^{-1}(a)$ contains no point corresponding to a polynomial with at most $3$ distinct roots.

\begin{figure}[h]
    \centering
    \begin{subfigure}[b]{0.45\textwidth}
        \includegraphics[width=\textwidth]{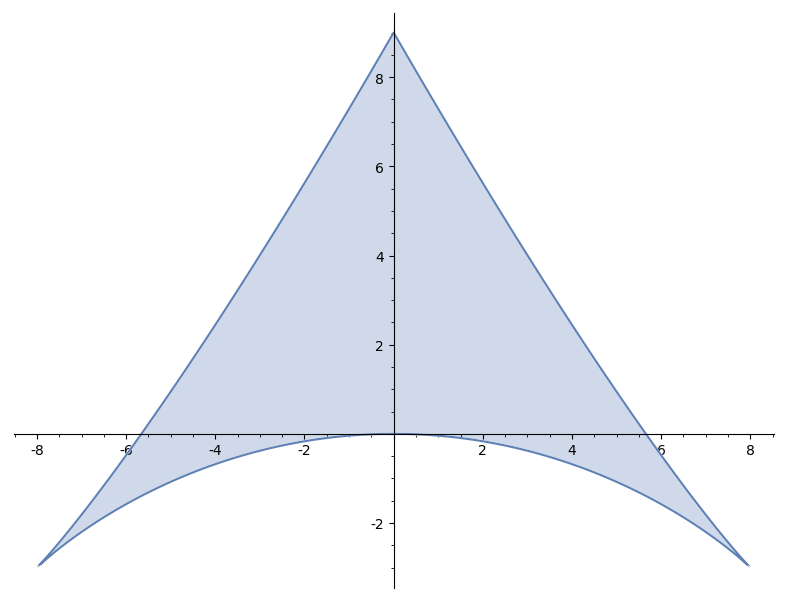}
        \caption{A hyperbolic slice. The convex hull is given by the convex hull of the three local extreme points.}
        \label{fig:hyperbolic}
    \end{subfigure}
    \hspace{0.05\textwidth}
    \begin{subfigure}[b]{0.45\textwidth}
        \includegraphics[width=\textwidth]{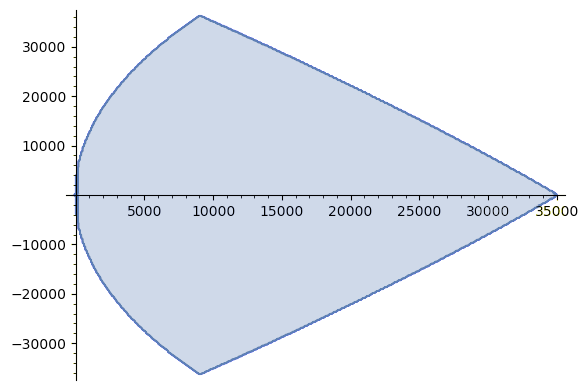}
        \caption{The stable slice $\mathcal{S} \cap L^{-1}(a)$.}
        \label{fig:slice1}
    \end{subfigure}
    \caption{Comparison of a hyperbolic slice and a stable slice.}
    \label{fig:combined}
\end{figure}
\end{example}

It often suffices to study stable slices of $\mathcal{S}_\HH(\mathbb{C})$ for a fixed half-plane $\HH$, since translations and rotations are linear isomorphisms. Let $\phi : \HH \to \G$ be a linear isomorphism between half-planes, and let $\psi = \phi^{-1}$ denote its inverse. Then $f_z \in \mathcal{S}_\HH(\C)$ if and only if $f_z \circ \psi \in \mathcal{S}_\G(\C)$.
In the following, we primarily work with $\HH = \HH_+$, the upper half-plane.

\begin{remark}
The set $\S_\HH(\mathbb{F})$ can be identified with a semialgebraic set in $\R^{2n}$. While the set of hyperbolic polynomials can be explicitly described using Sturm's Theorem (see e.g. \cite[Section~1.2]{theobald2024real}), it appears generally difficult to provide an explicit semialgebraic description of $\mathcal{S}_\HH(\mathbb{F})$. However, monic Hurwitz polynomials can be characterized as polynomials with a positive definite \emph{finite Hurwitz matrix} (see e.g. \cite[Section~9.3]{theobald2024real}). 
\end{remark}
The following remark establishes a connection between Hurwitz polynomials and stable polynomials, which we will use.
\begin{remark}\label{rem:hurwitz/stable}
The set of monic weakly Hurwitz polynomials $\HW$ can be embedded in $\S$ in the following way: If $f \in \HW$ is Hurwitz then the monic polynomial
\[
\tilde{f}(T)=(-i)^n       \cdot f(i\cdot T)  = T^n + \sum_{k=1}^ni^kz_kT^{n-k}
\]
is upper half-plane stable with coefficients alternating from the sets $\R$ or $i\cdot \R$. The map $\tilde ~ : \HW \to \S$ is linear, injective, not surjective, and its inverse is $g(T) \mapsto i^{n}g(-i\cdot T)$.
\end{remark}
\begin{example}
The polynomial \[f(T)=(T+2)(T+1+i)(T+1-i)=T^3+4T^2+6T+4\] is Hurwitz and \[\tilde{f}(T)=(-i)^{3}f(iT)=T^3-4iT^2-6T+4i\] is $\HH_+$-stable with alternating real and purely complex coefficients.   
\end{example}

\begin{definition}
An \emph{extreme point} of a convex set $C \subset  \C^n$ is a point that cannot be expressed as a convex combination of two distinct points in $C$.
For a set $A \subset \C^n$, we say that a point $z \in A$ is a \emph{local extreme point} of $A$ if there exists a neighborhood $U$ of $z$ such that $z$ is an extreme point of the convex hull $\conv(A \cap U)$.     
\end{definition}
The terminology \emph{local extreme point} is motivated by the real setting, where extreme points of a convex subset of $\R^n$ correspond to points where strict global minima or maxima of affine linear functions are attained. Similarly, local extreme points of $A \subset \R^n$ correspond to those points where strict local minima or maxima of such functions on $A$ are attained.

The following result complements the analysis of hyperbolic slices given in \cite[Theorem 4.2]{riener2012degree} and \cite[Theorem 2.8]{Riener-Schabert}, extending the discussion to the context of $\HH_+$-slices and real Hurwitz slices. As in the case of hyperbolic slices, the local extreme points in both of these settings correspond to polynomials with constrained root multiplicities.

\begin{theorem}\label{maintheorem}
Let $L : \C^n \to \C^k$ and $M : \R^n \to \R^k$ be surjective linear maps and $a \in \C^k$ and $b \in \R^k$ be points. 
\begin{enumerate}
    \item All local extreme points of an $\HH_+$-stable slice $\S \cap L^{-1}(a)$ are polynomials with at most $k$ roots in $\HH_+\setminus {\R}$ and at most $2k$ distinct real roots.
    \item All local extreme points of a real Hurwitz slice $\HW \cap M^{-1}(b)$ are polynomials with at most $k$ roots with negative real part and at most $2k$ distinct roots with real part equal to zero.
\end{enumerate}
\end{theorem}
The assumption that the linear maps are surjective is made for convenience. More generally, if $L$ (or $M$) is not surjective, the conclusion of Theorem \ref{maintheorem} still holds, with $k$ replaced by $\operatorname{rank}(L)$ (or $\operatorname{rank}(M)$). \\
Moreover, the following example shows that the converse of Theorem \ref{maintheorem} does not hold.

\begin{example}
    For $n=3$ and $k=1$, consider the surjective linear map $ L : \C^3 \to \C$ given by ${(z_1,z_2,z_3)} \mapsto {z_3}.$
    Then the point $(i,0,0) \in \S \cap L^{-1}(0)$ has one root in $\HH_+\setminus \R$ and one distinct real root. For all $\varepsilon \in (0,1)$ the points $((1-\varepsilon)i,0,0)$ and $((1+\varepsilon)i,0,0)$ also lie in $\S \cap L^{-1}(0)$. Since 
   \[(i,0,0)=\frac{((1-\varepsilon)i,0,0)+((1+\varepsilon)i,0,0)}{2}\] the point $(i,0,0)$ is not a local extreme point in the stable slice $\S \cap L^{-1}(0)$.
\end{example}
We present an example of a compact stable slice in Example \ref{ex:compact-slice}. The proofs of assertions (1) and (2) in Theorem \ref{maintheorem} follow the same line of arguments. Therefore, we isolate the key difference between them in Lemma \ref{lem:stability-perturbation}, which guarantees that small perturbations of certain \(D\)-stable polynomials remain \(D\)-stable. 

\begin{lemma}\label{lem:stability-perturbation}
\begin{itemize}
    \item [(a)] Let $\mathbb{F}\subset \C$, \(D\subset \C\) and \( p \in \mathbb{F} [T] \) be a monic polynomial whose roots lie entirely in the interior of \( D\) and \( h \in \mathbb{F}[T] \) be a polynomial with $\deg h < \deg p$. Then for all sufficiently small $\varepsilon >0$ the perturbed polynomials
    \[p_\varepsilon^{\pm} := (p \pm \varepsilon h)\]
    are \(D\)-stable, i.e., all their roots lie in \( D \).
    \item [(b)] Let \( q \in \R[T] \) be a monic hyperbolic polynomial with only distinct roots and \( h \in \R[T] \) be a polynomial with $\deg h < \deg q$. Then the perturbed polynomials
    \[q_\varepsilon^{\pm} := (q \pm \varepsilon h)\]
    are hyperbolic and therefore \(\HH_+ \)-stable for all sufficiently small $\varepsilon>0$.
    \item [(c)] Let \( q \in \R[T] \) be a monic polynomial of degree $m$ with only distinct roots and every root of $q$ has real part equal to zero. Let \( h(T)=\sum_{i=1}^mb_iT^{m-i} \in \R[T] \) be a polynomial which coefficients satisfy the linear conditions $b_{2i-1} = 0$, for $1 \leq i \leq \lfloor\frac{m}{2} \rfloor$.
    Then the perturbed polynomials
    \[q_\varepsilon^{\pm} := (q \pm \varepsilon h)\]
    have only roots with real part equal to zero for all sufficiently small $\varepsilon>0$.
\end{itemize}
\end{lemma}
\begin{proof}
    (a) and (b) follow immediately from the fact that the roots of a polynomial depend continuously on the coefficients (\cite{Roots}) and since complex roots of a real polynomial come as conjugated pairs. \\
    For (c), note that the weakly Hurwitz polynomial $q(T)$ corresponds to a monic polynomial $\tilde q (T) = (-i)^m q(i\cdot T) $ via the embedding stated in Remark \ref{rem:hurwitz/stable}. All roots of the polynomial $\tilde q$ are real, since every root of $q$ has real part equal to zero. By assumption, $q$ and thus $\tilde q$ have only pairwise distinct roots. In particular, $\tilde q$ is monic hyperbolic with only distinct roots. 
    Now, consider $h(T) = \sum_{i=1}^m b_iT^{m-i}\in \R[T]$ satisfying the condition stated in (c).
    Applying the same linear transform to $h$ gives the polynomial $\Bar{h}(T) := (-i)^m h(i\cdot T) = \sum_{k=1}^m (-1)^m b_i i^{2m-k}T^{m-k} \in \C[T]$ but the condition on the coefficients of $h$ guarantees $\Bar{h} \in \R[T]$. By (b) we have that the perturbed polynomials $r_\varepsilon^\pm := \tilde q \pm \varepsilon\Bar{h}$ are hyperbolic for all sufficiently small $\varepsilon > 0$. Then for all small $\varepsilon > 0$ the polynomials $i^mr_\varepsilon^\pm(-i \cdot T) = i^m \tilde q (-i \cdot T) \pm \varepsilon \cdot i^m \Bar{q} (-i \cdot T) = q(T)\pm \varepsilon h(T) $ have only distinct roots and every root has real part equal to zero. This was to show.
\end{proof}

Now we can prove the characterization of local extreme points in stable slices.

\begin{proof}[Proof of Theorem~\ref{maintheorem}] 
\begin{enumerate}
    \item We begin with proving assertion (1). 
Let \( z \in \S \cap L^{-1}(a) \) be a local extreme point; that is, there exists a neighborhood \( U \) of \( z \) such that \( z \) is an extreme point of \( \conv(\S \cap L^{-1}(a) \cap U) \). Define the monic polynomial
\[
f(T) := T^n - z_1 T^{n-1} + \dots + (-1)^n z_n,
\]
and factor \( f = p \cdot r \), where \( p \) has only roots in \( \HH_+ \setminus \R \) and \( r \) has only real roots.

\begin{enumerate}
\item We first show that \( \deg(p) \leq k \). Suppose for contradiction that \( \deg(p) = m > k \). Write
\[
r(T) = T^{n - m} + r_1 T^{n - m - 1} + \dots + r_{n - m}, \quad \text{with } r_0 := 1,
\]
and define the linear map
\[
\chi : \C^m \to \C^n, \quad y \mapsto \left( \sum_{i+j=1} r_i y_j, \dots, \sum_{i+j=n} r_i y_j \right),
\]
where \( 0 \leq i \leq n - m \), \( 1 \leq j \leq m \), and \( i + j \) ranges from \( 1 \) to \( n \). Since \( m > k \), we can find a nonzero vector \( b \in \ker(L \circ \chi) \).

Define the perturbation polynomial
\[
h(T) := b_1 T^{m-1} + \dots + b_m, \quad \text{and } g(T) := h(T) \cdot r(T) = c_1 T^{n-1} + \dots + c_n,
\]
where \( c = \chi(b) \in \ker L \). By Lemma~\ref{lem:stability-perturbation}~(a), the perturbed polynomials \( f \pm \varepsilon g = (p \pm \varepsilon h) \cdot r \) are stable for all \( \varepsilon > 0 \) sufficiently small. Hence, the coefficient vectors \( z \pm \varepsilon c \) lie in \( \S \cap L^{-1}(a) \). For small enough \( \varepsilon \), these vectors also lie in \( U \), so we have
\[
z = \frac{(z + \varepsilon c) + (z - \varepsilon c)}{2},
\]
contradicting the extremality of \( z \). Therefore, \( \deg(p) \leq k \).

\item Now we show that \( r \) has at most \( 2k \) distinct real roots. Suppose for contradiction that \( r \) has \( m > 2k \) distinct real roots \( x_1, \dots, x_m \). Factor
\[
f = q \cdot s, \quad \text{where } q(T) := \prod_{i=1}^m (T - x_i), \quad \deg(s) = n - m.
\]
Write
\[
s(T) = T^{n - m} + s_1 T^{n - m - 1} + \dots + s_{n - m}, \quad \text{with } s_0 := 1,
\]
and define the linear map
\[
\chi : \R^m \to \C^n, \quad y \mapsto \left( \sum_{i+j=1} s_i y_j, \dots, \sum_{i+j=n} s_i y_j \right).
\]
Since \( m > 2k \), we can again find a nonzero vector \( b \in \ker(L \circ \chi) \), and define
\[
h(T) := b_1 T^{m-1} + \dots + b_m, \quad g(T) := h(T) \cdot s(T) = c_1 T^{n-1} + \dots + c_n,
\]
with \( c = \chi(b) \in \ker L \). Since \( q \) has only simple real roots, it is hyperbolic, and by Lemma~\ref{lem:stability-perturbation}~(b), the perturbed polynomials \( f \pm \varepsilon g = (q \pm \varepsilon h) \cdot s \) are again stable for small enough \( \varepsilon > 0 \). Thus, \( z \pm \varepsilon c \in \S \cap L^{-1}(a) \cap U \), and again
\[
z = \frac{(z + \varepsilon c) + (z - \varepsilon c)}{2},
\]
contradicting extremality. Therefore, \( r \) has at most \( 2k \) distinct real roots.
\end{enumerate}
\item The proof of assertion (2) proceeds analogously to that of assertion (1), but relies on Lemma \ref{lem:stability-perturbation} (c) instead of Lemma \ref{lem:stability-perturbation} (b).
\end{enumerate}
\end{proof}

\begin{remark}[Unified Local Extremality Criterion]
The theorem above addresses only the cases of $\mathbb{H}_+$-stable and Hurwitz-stable polynomials. However, we note that the proof extends to arbitrary slices, with a small caveat: Let \( L : \C^n \to \C^k \) be a surjective linear map, and suppose that \( z \in \mathcal{S}_D(\C) \cap L^{-1}(a) \) is a local extreme point. Then, following the same reasoning, the associated polynomial \( f_z \) satisfies the following: it has at most \( k \) roots in \( \mathrm{int}(D) \setminus \R \), and at most \( 2k \) distinct real roots if \( \R \) is a subset of the boundary of \(D \).
\end{remark}

Moreover, as a corollary of Theorem \ref{maintheorem}, we obtain a result for arbitrary stable slices for closed half-planes.
\begin{corollary}\label{cor:maintheorem}
    Let $\HH$ be a closed half-plane. Every non-empty $\HH$-stable slice $\S_\HH \cap L^{-1}(a)\neq \emptyset$ contains a point that corresponds to a polynomial with at most $k+2$ roots in the interior of $\HH$ and at most $2(k+2)$ distinct roots in the boundary of $\HH$.
\end{corollary}
Note that the result in Corollary \ref{cor:maintheorem} does not depend on the degree $n$ of the univariate polynomials, making it particularly interesting when $n$ is large. In fact, we observe a form of stabilization in the structure of local extreme points of stable slices once the number of variables exceeds $3k$.

Before proving Corollary \ref{cor:maintheorem}, we examine the set of stable polynomials of degree $n$ with fixed leading coefficients, which constitutes a specific instance of a stable slice. We will show that these slices are compact (if $k \geq 2$) which will be instrumental in ensuring the existence of local extreme points in stable slices. 
\begin{definition}
For an integer $1 \leq k \leq n$ and a point $a=(a_1,\ldots,a_k) \in \C^k$ we define $\S (a) = \S \cap \{ z \in \C^n : z_1 = a_1, \ldots, z_k = a_k \}$ as the set of all monic $\HH_+$-stable polynomials of degree $n$ whose first $k$ non-trivial coefficients are determined by the point $a$. 
\end{definition}
With the notation introduced earlier, we have $\S(a) = \S \cap L^{-1}(a)$, where $L : ~ \C^n \to \C^k$ denotes the projection onto the first $k$ coordinates. 
The following lemma guarantees that stable slices $\S(a)$ have a local extreme point and thus contain a polynomial whose roots are distributed as described in Theorem \ref{maintheorem}.

\begin{lemma}\label{compact}
For an integer $2 \leq k \leq n$ the stable slice $\S(a)$ is compact. 
\end{lemma}
\begin{proof}
As the empty set is compact we can assume that there exists a point $z \in \S(a)$. Furthermore, we denote by $x=(x_1,\dots,x_n)\in \HH_+$ the roots of the polynomial
\[f_z := T^n - z_1 T^{n-1} + \ldots +(-1)^n z_n.\]
Then, if $e_1$ and $e_2$ denote the first and second elementary symmetric polynomial in $n$ variables, we have
\[\sum_{i=1}^n x_i = e_1(x) = a_1\]
and hence the imaginary part of the $x_i's$ is contained in $[0,\Im(a_1)]$. Furthermore, we have
\begin{align}\label{eq:p2 in e1 and e2}
\sum_{i=1}^n x_i^2 = e_1(x)^2 - 2e_2(x) = a_1^2-2a_2
\end{align}
and hence
\begin{align*}
\sum_{i=1}^n \Re(x_i)^2 = \sum_{i=1}^n \Re(x_i^2) + \Im(x_i)^2 \leq \sum_{i=1}^n \Re(x_i^2) + \Im(a_1)^2 
= \Re\left( \sum_{i=1}^n x_i^2 \right) +n\Im(a_1)^2~ .
\end{align*}
Then using the identity in (\ref{eq:p2 in e1 and e2}) we obtain 
\[\sum_{i=1}^n \Re(x_i)^2  \leq \Re(a_1^2-2a_2)+n\Im(a_1)^2 ~ .\]

This shows that also the real part of the $x_i$'s is bounded. Thus the set $\S(a)$ is bounded. Furthermore, as the roots of a polynomial depend continuously on the coefficients it is clear that $\S(a)$ is closed and therefore compact.
\end{proof}

The following example shows that the converse of Theorem~\ref{maintheorem} also does not necessarily hold for compact stable slices.
\begin{example}\label{ex:compact-slice}
Let  $L : \C^{10} \to \C^2$ be the linear projection onto the first $2$ coordinates.
The $\HH_+$-stable polynomial $f_z(T)=  (T-1-i)(T+1-i) (T-2)(T+2)(T-1)(T+1)^5$ has two roots in $\HH_+ \setminus \R$ and $4$ distinct real roots. 
Thus the point $z=(1+i,-1+i,2,-2,1,-1,\ldots,-1)$ lies in the stable slice $\S(-4+2\cdot i, -1-8\cdot i)$ which is compact by Lemma \ref{compact}. We factor $f_z(T) = g(T) \cdot h(T)$, where $h(T)=(T-2)(T+2)(T-1)$. By Lemma \ref{lem:stability-perturbation} the perturbed polynomials $h_\varepsilon^\pm (T) = h(T) \pm \varepsilon \cdot 1 $ are real rooted for small $\varepsilon > 0$.
Thus $g \cdot h_\varepsilon^\pm = f_z \pm \varepsilon g$ is $\HH_+$-stable for sufficiently small $\varepsilon > 0$ and lies in the same stable slice.
Therefore $z$ is not a local extreme point which shows that the converse of Theorem~\ref{maintheorem} fails also in the compact case. 
\end{example}

\begin{remark}\label{rem:compact}
 For a surjective linear map $L :~ \C^n \to \C^k$ and a point $a \in \C^k$ the stable slice $\S \cap L^{-1}(a)$ may be unbounded. Then we can consider the linear map $\widetilde{L} : ~ \C^n \to \C^{k+2}$, where $\widetilde{L}(z)=(L(z),z_1,z_2)$. The set $\S \cap \widetilde{L}^{-1}(b)$ is compact for any point $b \in \C^{k+2}$, by a similar argument as in the proof of Lemma \ref{compact}. Moreover, if one or both of the first two unit vectors are in the row span of a matrix representation of $L$, then we can consider $\widehat{L}(z)=(L(z),z_j)$ for $j \in \{1,2\}$ instead of $L$ or the original stable slice was already compact.
\end{remark}

\begin{proof}[Proof of Corollary \ref{cor:maintheorem}]
    Since $\HH$ can be bijectively mapped to $\HH_+$ under a linear isomorphism it suffices to show the theorem for $\HH=\HH_+$. Now the claim follows from Theorem \ref{maintheorem}, Lemma \ref{compact} and Remark \ref{rem:compact}.
\end{proof}

\begin{remark}\label{rem:case_upper_halfplane}
    In the case that $L$ (or $M$) is the projection to the first $k<n$ coordinates, we can replace $2k$ by $k$ in Theorem \ref{maintheorem}. This is, since $(0,\ldots,0,1) \in \ker (L \circ \chi)$ and we can choose $h(T):=1$ in the proof in this case. Moreover, if $k\geq 2$ the stable slice is compact by Lemma \ref{compact}. So we can say that every such non-empty stable slice contains a point corresponding to a polynomial with at most $k$ roots in the interior of $\mathbb{H}_+$ and at most $k$ distinct roots in the boundary of $\HH_+$.
\end{remark}

\section{A Grace-Walsh-Szeg\H{o} like theorem for symmetric polynomials in few multiaffine polynomials} \label{Sec:Grace-Walsh-Szego}
Throughout this section, let $\HH$ be a closed half-plane, let $\X := (X_1,\ldots,X_n)$ be a tuple of $n$ variables and let $\C[\underline{X}]$ and $\R[\underline{X}]$ denote the complex and real polynomial rings in $\underline{X}$. 

The main results of this section are a statement akin to the well-known Grace–Walsh–Szegő coincidence theorem (Theorem \ref{thm:grace}) and a generalization of the degree principle (Theorem \ref{degree-principle}), which we derive from Proposition \ref{prop:main2}. This proposition, in turn, follows from the characterization of local extreme points of stable slices via root multiplicities given in Theorem \ref{maintheorem}. We present the proofs in Subsection \ref{subsec:proofs} and in Subsection \ref{subsec:converse} we briefly discuss an alternative way of generalization of Grace-Walsh-Szeg\H{o}.  \\

Recall that a multivariate polynomial is called \textit{multiaffine}, if it is linear in every variable.

\begin{theorem}[Grace-Walsh-Szeg\H{o} coincidence theorem, \cite{grace1902zeros}] \label{grace-walsh-szego}
Let $D$ be a circular region and let $f \in \C[\X]$ be a multiaffine symmetric polynomial. If $\deg (f) = n$ or if $D$ is convex, then for any $(x_1,\ldots,x_n) \in D^n$ there exists a $y \in D$ with $f(x_1,\ldots,x_n)=f(y,\ldots,y)$.
\end{theorem}

Recall that every symmetric polynomial in \( n \) variables can be uniquely expressed as a polynomial in the first \( n \) elementary symmetric polynomials. This is known as the \emph{fundamental theorem of symmetric functions} (see, e.g., \cite[p.~20]{macdonald1998symmetric}). 

In the sequel, we introduce new variables \( \underline{Z} := (Z_1, \ldots, Z_n) \) and note that for any symmetric polynomial \( f \in \mathbb{F}[\X] \), where \( \mathbb{F} \in \{\mathbb{R}, \mathbb{C}\} \), there exists a unique polynomial \( g \in \mathbb{F}[\underline{Z}] \) such that
\[
f(\X) = g(e_1(\X), \ldots, e_n(\X)).
\]

In particular, we are interested in symmetric polynomials that depend only on a small number of linear combinations of elementary symmetric polynomials. Multiaffine symmetric polynomials fall into this category, as they can be described as one affine linear combination of \( e_1, \ldots, e_n \). However, polynomials such as \( (2e_1 + 3e_4 + i)^d \) illustrate that the degree can still be arbitrarily large, even when the dependence is restricted to few combinations.

We now introduce the following notation to describe points with controlled numbers of distinct coordinates.

\begin{definition} \label{def:univariate definitions}
Let \( \HH \) be a closed half-plane. 
We define \( \HH_{k,m} \subset \HH^n \) to be the set of points with at most \( k \) distinct coordinates on the boundary of \( \HH \) and at most \( m \) coordinates in the interior:
\[
\HH_{k,m}= \left\{ x\in \HH^n ~:~  |\{x_1,\dots,x_n\}\cap \operatorname{bd} \HH|\leq k \text{ and } |\{j \in \{1,\dots,n\} : x_j \in \operatorname{int} \HH \}| \leq m \right\}.
\]
\end{definition}

The classical Grace–Walsh–Szegő theorem guarantees that symmetric multiaffine polynomials attain their values at diagonal points. In this work, we extend this perspective to symmetric polynomials depending on few multiaffine building blocks. Using the structure of stable slices and local extremality developed in Section 1, we show that a similar, though weaker, coincidence property holds: function values can be matched at points with few distinct coordinates, explicitly controlled by the algebraic structure of the polynomial.

\begin{theorem}\label{thm:grace}
    Let $f\in \C[\X]$ be a symmetric polynomial that can be written as a polynomial in $k$ symmetric and multiaffine polynomials. Furthermore, let $x\in \HH^n$. Then there exists a point $\tilde x \in \HH_{2(k+2),k+2}$ with
    $ f(x)=f(\tilde x). $
\end{theorem}

Note that, in contrast to the Grace–Walsh–Szeg\H{o} coincidence theorem, our result does not require the polynomial $f$ to be multiaffine. However, it is also less general in two respects: first, we restrict our attention to half-planes rather than arbitrary circular regions; second, the conclusion (in the multiaffine case) we obtain is weaker.
More precisely, if $f$ is symmetric and multiaffine of degree $d\geq 2$ and $x\in D^n$, then our result guarantees the existence of a point $\tilde x \in \HH_{6,3}$ such that
    $ f(x)=f(\tilde x), $
    whereas the Grace-Walsh-Szeg\H{o} theorem ensures the existence of a $y \in \HH$ such that $ f(x) = f(y,\ldots,y) .$ 

Moreover, we address the problem of giving a \emph{degree principle} in analogy to \emph{Timofte's degree principle} for real varieties.
\begin{theorem}[Degree principle, \cite{timofte2003positivity}]
Let $f_1,\ldots,f_m \in \R[\underline{X}]$ be symmetric polynomials. Then their real variety is non-empty if and only if it contains a real point with at most $\max_{1 \leq i \leq m}\{\deg (f_i),2\}$ many distinct coordinates.
\end{theorem}

\begin{definition}\label{def:stable}
Let \( f_1, \ldots, f_m \in \mathbb{C}[\underline{X}] \) be polynomials, and denote by \( V(f_1, \ldots, f_m) \subset \mathbb{C}^n \) their complex zero set. 
Let \( \HH \subset \mathbb{C} \) be a closed half-plane. 
We say that the variety \( V(f_1, \ldots, f_m) \) is \emph{disjoint from \( \HH^n \)} if
\[
V(f_1, \ldots, f_m) \cap \HH^n = \emptyset.
\]
Similarly, we say that a polynomial \( f \in \mathbb{C}[\underline{X}] \) is \emph{\( \HH \)-disjoint} if its zero set \( V(f) \) is disjoint from \( \HH^n \).
\end{definition}

\begin{remark}
In Definition \ref{def:stable}, we use the notion of \(\HH\)-disjointness for multivariate polynomials, meaning that the zero set does not intersect \(\HH^n\). 
In the literature a multivariate polynomial which is $\HH$-disjoint is usually called stable.
This contrasts with the notion of stability for univariate polynomials in Definition \ref{def:slices}, where a univariate polynomial is called \(\HH\)-stable if all of its roots lie inside \(\HH\).
Since the complement \(\HH^c\) of \(\HH\) in \(\mathbb{C}\) is an open half-plane, it follows that for univariate polynomials, \(\HH\)-disjointness corresponds to \(\HH^c\)-stability.
\end{remark}

\begin{theorem}[Double-degree principle]\label{degree-principle}
    Let $f_1,\dots, f_m \in \C[\X]$ be symmetric polynomials of degree at most $d$. Then 
    \[V(f_1,\dots, f_m)\cap \HH^n = \emptyset \iff V(f_1,\dots,f_m)\cap \HH_{2(d+2),d+2} = \emptyset. \]
    Moreover, if $\HH$ is a rotation of the upper half-plane we can replace $\HH_{2(d+2),d+2}$ by $\HH_{d,d}$.
\end{theorem}
Although one might hope for a stronger degree principle, the next example shows that disjointness of a variety defined by symmetric polynomials of degree $\leq d$ cannot always be checked by testing points with at most $d$ many distinct coordinates.

\begin{example}
Let $n=4$ and consider $f_1:=e_1-23 i$, $f_2:=e_2-463 i$ and $f_3:=e_3-8461 i$. Then
\[V(f_1,f_2,f_3)\cap \HH_+^4 \neq \emptyset ~ \text{and} ~ V(f_1,f_2,f_3)\cap \{ x\in \HH_+^4 ~:~ |\{x_1,\dots,x_4\}|\leq 3 \} = \emptyset,\]
which can either be computed directly using a Gröbner basis or concluded by using Example \ref{ex:stable-slice}.
\end{example}

\begin{remark}
    The results also extend to open half-planes as follows: let \( \G \subset \mathbb{C} \) be an open circular region, and let \( x \in \G^n \). Then \( x \in \HH^n \) for some closed half-plane \( \HH \subset \G \). 
    Consequently, in Proposition \ref{prop:main2} and Theorem \ref{thm:grace}, the set \( \G_{2(k+2),k+2} \) can be replaced by \( \G_{0,3(k+2)} \), and in Theorem \ref{degree-principle}, \( \G_{2(d+2),d+2} \) can be replaced by \( \G_{0,3(d+2)} \).
\end{remark}

If $\HH=\HH_+$ is the upper half-plane, one can also formulate a generalization of the half-degree principle for the upper half-plane.

\begin{theorem}[Half-degree principle for the upper half-plane]\label{thm:half-degree}
Let $f\in \C[\X]$ be a symmetric polynomial of degree $d\leq n$ and $\lambda,\mu \in \R$. Then
\[\inf_{x\in \HH_+^n} \lambda \Re(f(x)) + \mu \Im(f(x)) =\inf_{x\in {\HH_+}_{k,k}} \lambda \Re(f(x)) + \mu \Im(f(x)),\]
where $k=\max\{\lfloor \frac{d}{2} \rfloor,2\}$.
\end{theorem}

\subsection{Proofs of Theorems \ref{thm:grace}, \ref{degree-principle} and \ref{thm:half-degree}}  \label{subsec:proofs}
The proofs of Theorems \ref{thm:grace} and \ref{degree-principle} are based on our result concerning the existence of polynomials with few distinct roots in stable slices (see Theorem \ref{maintheorem}).

\begin{definition}\label{def:sufficient}
Let $f \in \C[\X]$ be a symmetric polynomial and write $f(\underline{X})=g(e_1(\underline{X}),\ldots,e_n(\underline{X}))$ in terms of elementary symmetric polynomials for a unique polynomial $g \in \C[\underline{Z}]$. 
\begin{enumerate}
    \item We say that $f$ is \emph{$(\ell_1,\dots,\ell_k)$-sufficient} if $g\in\C[\ell_1,\dots,\ell_k]$ where $\ell_1(\underline{Z}),\dots,\ell_k(\underline{Z})$ are linear forms.
    \item Moreover, we say that an algebraic  variety $V\subset \C^n$, which is closed under the permutation action of $S_n$, is \emph{$(\ell_1,\dots,\ell_k)$-sufficient}, if $V$ can be described as vanishing set of $(\ell_1,\dots,\ell_k)$-sufficient polynomials $f_1,\ldots,f_m$.
\end{enumerate} 
\end{definition}

\begin{remark}\label{rem:multiaffine_sufficient}
A polynomial \( f \) is called \emph{\((\ell_1,\dots,\ell_k)\)-sufficient} for some linear forms \( \ell_1,\dots,\ell_k \) if and only if \( f \) can be expressed as a polynomial in \( k \) symmetric and multiaffine polynomials. 

In particular, every symmetric and multiaffine polynomial is \(\ell_1\)-sufficient for some linear form \(\ell_1\), and any symmetric polynomial \( f \in \mathbb{C}[\underline{X}] \) is \((Z_1,\ldots,Z_n)\)-sufficient.

For instance, for \( n \geq 3 \), the polynomial \( e_1^2(\underline{X}) + e_2(\underline{X}) + 2e_3(\underline{X}) \) is \((\ell_1, \ell_2)\)-sufficient, where \( \ell_1(\underline{Z}) = Z_1 \) and \( \ell_2(\underline{Z}) = Z_2 + 2Z_3 \). 
For further details on the notion of sufficiency and methods for checking sufficiency, we refer to \cite[Subsection 3.3]{Riener-Schabert}.
\end{remark}

The following lemma is an immediate consequence of the fundamental theorem of symmetric functions and may serve as a motivation for Definition \ref{def:sufficient}. 

\begin{lemma}\label{rem:degree}
Let $f \in \R[\X]$ be a symmetric polynomial of degree $d \leq n$. Then $f$ is $\left( Z_1,\dots,Z_{d} \right)$-sufficient, i.e., $f$ can be written as $f=g(e_1,\dots,e_d)$ for some $g\in \C[Z_1,\dots,Z_d]$. Moreover, $g$ is linear in $Z_{\lfloor\frac{d}{2}\rfloor+1},\dots,Z_d$.
\end{lemma}

Theorems \ref{thm:grace} and \ref{degree-principle} will follow as consequences of the following key result.

\begin{proposition}\label{prop:main2}
    Let \( V \subset \mathbb{C}^n \) be an algebraic set defined by \((\ell_1,\dots,\ell_k)\)-sufficient symmetric polynomials. 
    Then \( V \) is \(\HH\)-disjoint if and only if 
    \[
    V \cap \HH_{2(k+2),k+2} = \emptyset.
    \]
\end{proposition}

\begin{proof}
    The forward implication is immediate from the definitions. 
    
    For the converse, assume that \( V \) is not \(\HH\)-disjoint. Then there exists a point \( x \in V \cap \HH^n \). 
    Consider the vector \( z := (e_1(x), \dots, e_n(x)) \in \mathbb{C}^n \), which lies in the set \( \mathcal{S}_{\HH}(\mathbb{C}) \cap L^{-1}(a) \), where
    \[
    L : \mathbb{C}^n \to \mathbb{C}^k, \quad y \mapsto (\ell_1(y), \dots, \ell_k(y)),
    \quad \text{and} \quad a := L(z) \in \mathbb{C}^k.
    \]
    By Corollary \ref{cor:maintheorem}, there exists a point \( \tilde{z} \in \mathcal{S}_{\HH}(\mathbb{C}) \cap L^{-1}(a) \) such that the associated polynomial \( f_{\tilde{z}} \) has roots \( \tilde{x} \in \HH_{2(k+2),k+2} \).
    In particular, \( \tilde{z} = (e_1(\tilde{x}), \dots, e_n(\tilde{x})) \) and \( L(\tilde{z}) = a = L(z) \).

    Since \( V \) is \((\ell_1,\dots,\ell_k)\)-sufficient, and \( \tilde{z} \) satisfies the same defining relations as \( z \), it follows that \( \tilde{x} \in V \).
    Thus \( V \cap \HH_{2(k+2),k+2} \neq \emptyset \), completing the proof.
\end{proof}

\begin{proof}[Proof of Theorem \ref{thm:grace}]
The statement follows directly from Remark \ref{rem:multiaffine_sufficient} and Proposition \ref{prop:main2}.
\end{proof}

\begin{proof}[Proof of Theorem \ref{degree-principle}]
The first part of the theorem follows immediately from Proposition \ref{prop:main2} and Lemma \ref{rem:degree}. \\
In the case that $\HH$ is a rotation of the upper half-plane we can replace $\HH_{2(d+2),d+2}$ by $\HH_{d,d}$. This follows from Remark \ref{rem:case_upper_halfplane} for $d\geq 2$ and the case $d=1$ is trivial. 
\end{proof}

\begin{proof}[Proof of Theorem \ref{thm:half-degree}]
    Write $f=g(e_1,\dots,e_d)$ for some $g\in\C[\underline{Z}]$ and observe that $g$ is linear in $Z_{\lfloor\frac{d}{2}\rfloor+1},\dots,Z_d$ by Lemma \ref{rem:degree}.
    Let now $x\in \HH_+^n$ and consider $z:=(e_1(x),\dots,e_n(x))\in \S(a)$, where $a:=(e_1(x),\dots,e_k(x))$. Since $\S(a)$ is compact and $g$ is linear on $\S(a)$, the minimum of $\lambda \Re (g) + \mu \Im (g)$ on $\S(a)$ is taken on an extreme point of the convex hull of $\S(a)$, i.e., on a point $\tilde z\in \S(a)$ which corresponding polynomial $f_{\tilde{z}}$ has roots $x \in {\HH_+}_{k,k}$ by Remark \ref{rem:case_upper_halfplane}.
\end{proof}

\subsection{A converse to Grace-Walsh-Szeg\H{o}'s coincidence theorem} \label{subsec:converse}

In a contrasting direction of generalization, Brändén and Wagner \cite{Bränden-Wagner} showed that for the open upper half-plane $\operatorname{int} \HH_+$ and for any proper subgroup $G \subsetneq S_n$ acting on $\C[\underline{X}]$ by variable permutation, there is no analogue of the Grace-Walsh-Szeg\H{o} coincidence theorem.
\begin{theorem}\cite[Theorem~2]{Bränden-Wagner}
Let $G \subset S_n$ be a permutation group. 
Suppose that for any multiaffine
$G$-invariant polynomial $f \in \C[\underline{X}]$ and any  $x \in \operatorname{int} \HH_+^n$ there is a $y \in \operatorname{int} \HH_+$ with $f(y,\ldots,y)=f(x)$, then $G$ must be already the full symmetric group $S_n$.
\end{theorem}

By considering Young subgroups of $S_n$ we find that a weaker statement still holds.
Recall that a \emph{Young subgroup} of $S_n$ is a group which is isomorphic to $S_{\lambda_1} \times \cdots \times S_{\lambda_l}$ for a partition $(\lambda_1,\ldots,\lambda_l)$ of $n$.

\begin{definition}\label{def:Youngsubgroup}
    For a group $G \subset S_n$ we write $S(G)=\tilde{S}_{i_1}^1\times \dots \times \tilde{S}_{i_{\kappa(G)}}^{\kappa(G)} \subset S_n$ for
    a Young subgroup of $G$, where $\tilde{S}_{i_j}^j$ is the symmetric group on $i_j$ elements acting on $\C^n$ by permuting the $i_1+\dots+i_{j-1}+1$ to $i_1+\dots+i_{j}$-th coordinates and $ \kappa(G)$ be the minimal number of factors needed to define such a Young subgroup of $S_n$.
\end{definition}
For instance, for $G = \langle (i,i+1) : 1 \leq i \leq n-2 \rangle$ we have $\kappa(G)=2$ and $S(G) = \tilde{S}_{n-1}^1 \times \tilde{S}_1^2$ is a Young subgroup defined in Definition \ref{def:Youngsubgroup}. 
Observe that $\sum_{j=1}^{\kappa(G)} i_j=n$ holds for all groups $G\subset S_n$.
\begin{proposition}
    Let $\HH$ be a half-plane, $f\in \C[\X]^G$ be a $G$-invariant multiaffine polynomial and $x\in \HH^n$. Then there are $y_1,\dots,y_{\kappa(G)} \in \HH$, such that
    \[f(x)=f(\underbrace{y_1,\dots,y_1}_{i_1 \text{-times}},\dots,\underbrace{y_{\kappa(G)},\dots,y_{\kappa(G)}}_{i_{\kappa(G)} \text{-times}} ).\]
    
\end{proposition}
\begin{proof}
Let $S(G)=\tilde{S}_{i_1}^1\times \dots \times \tilde{S}_{i_{\kappa(G)}}^{\kappa(G)}\subset S_n$ be as in Definition \ref{def:Youngsubgroup} and $x=(x_1,\ldots,x_n)\in \HH^n$. The polynomial
\[f_1:=f(X_1,\dots,X_{i_1},x_{i_1+1},\dots,x_n)\subset \C[X_1,\dots,X_{i_1}]\]
is $\tilde{S}_{i_1}^1$-invariant and multiaffine, so by Grace-Walsh-Szeg\H{o}'s coincidence theorem, there is $y_1\in \HH$, such that
\[f(x)=f_1(x_1,\dots,x_{i_1})=f_1(\underbrace{y_1,\dots,y_1}_{i_1 \text{-times}}).\]
Recursively, we define the $\tilde{S}_{i_j}^j$-invariant polynomial \[f_j=f(\underbrace{y_1,\dots,y_1}_{i_1 \text{-times}},\dots,\underbrace{y_{j-1},\dots,y_{j-1}}_{i_{j-1} \text{-times}},X_1,\dots,X_{i_j},x_{i_1+\dots+i_{j}+1},\dots,x_{n})\]
and, by Grace-Walsh-Szeg\H{o}'s theorem, there is a $y_j\in D$, such that
\[f(x)=f_j(x_{i_1+\dots+i_{j-1}+1},\dots, x_{i_1+\dots+i_{j}})=f_j(\underbrace{y_j,\dots,y_j}_{i_j \text{-times}}).\]
\end{proof}
Using the result of Bränd\'en and Wanger we can formulate the following converse statement:

\begin{proposition}
Let $G\subset S_n$ and $H=\tilde{S}_{j_1}^1\times \dots \times \tilde{S}_{j_m}^m \subset S_n$
be a supergroup of $G$.
If for any $G$-invariant multiaffine polynomial $f\in \C[\X]^G$ and any $x\in (\operatorname{int} \HH_+)^n$, there are $y_1,\dots,y_m\in \operatorname{int} \HH_+$, such that
    \[f(x)=f(\underbrace{y_1,\dots,y_1}_{j_1 \text{-times}},\dots,\underbrace{y_m,\dots,y_m}_{j_m \text{-times}} ),\]
then every $G$-invariant multiaffine polynomial is already $H$-invariant.
\end{proposition}

\section{Conclusion and open questions} \label{sec:final}
In this work, we restricted our attention to half-plane stable polynomials. However, the notion of stability can be formulated for any domain or more typically for any circular region.
 It is well known that Möbius transformations map circular regions to circular regions, and that testing stability for an arbitrary circular region can be reduced to testing $\HH_+$-stability for an associated polynomial, possibly of smaller degree.
 Specifically, let $D$ be a circular region and let $\phi (z) = \frac{az+b}{cz+d}$ be a Möbius transformation mapping $\HH_+$ to $D$. Then a monic polynomial $f \in \C[T]$ is $D$-stable if and only if the polynomial $(cT+d)^{\deg (f)}f\left(\frac{aT+b}{cT+d}\right)$ is $\HH_+$-stable. The roots of the associated polynomial are contained in the image of the roots of $f$ under $\phi^{-1}$. However, the transformed polynomial may fail to be monic or may have lower degree. This phenomenon occurs when a root of $f$ is mapped to the pole of $\phi^{-1}$. For instance, consider $f=p\cdot (T-1)$ which is $\{ x \in \C : |x|\leq 1\}$-stable and $p$ has not root at $1$, then 
\[(T+i)^{\deg{(p)}}p\left(\frac{T-i}{T+i}\right)(T+i)\left(\frac{T-i}{T+i}-1\right)=(T+i)^{\deg{(p)}}p\left(\frac{T-i}{T+i}\right)\cdot (-2i)\]
is a non-monic $\HH_+$-stable polynomial of degree $\deg(f)-1$. Thus our proof of Theorem \ref{maintheorem} does not transfer to circular regions which are bounded by a circle. Nevertheless, the following questions seem worth to be asked.
\begin{question}
    Can Theorem \ref{maintheorem} be adapted to arbitrary circular regions? If not, can our variation of the coincidence theorem be extended to a closed domain bounded by a circle?
\end{question}

\begin{question}
   Can the double-degree principle from Theorem~\ref{degree-principle} be further improved?
\end{question}

The Grace-Walsh-Szeg\H{o} coincidence theorem plays a central role in understanding the stability of multivariate polynomials. One important consequence is that a polynomial is stable if and only if its polarization, which is multiaffine and symmetric in each group of variables, is stable (see, e.g., \cite[Section~9.4]{theobald2024real}). The polarization of a polynomial $f \in \C[\underline{X}]$ introduces new groups of variables (one for each variable of $f$).

In this work, we have established a result akin to the Grace-Walsh-Szeg\H{o} theorem that applies to certain non-multiaffine symmetric polynomials. This raises the natural question:

\begin{question} Can Theorem~\ref{thm:grace} be used to characterize or construct classes of linear operators that preserve stability for (non-multiaffine symmetric) polynomials? \end{question}

Finally, it might be interesting to study possible stratifications of the set of weakly Hurwitz polynomials with respect to root multiplicities, similar to the stratifications investigated for hyperbolic polynomials by Arnold \cite{Arnold}, Kostov \cite{Kostov}, Meguerditchian \cite{Meguerditchian}, and others. Recently, Lien \cite{lien2023hyperbolic} showed that for the special case of fixed first $k$ coefficients in the hyperbolic setup, one can reconstruct the stratification's compositions from those of its $0$-dimensional strata, and Schabert and Lien \cite{Schabert-Lien} demonstrated that in this case the resulting poset has a structure similar to that of a polytope, leading to the same bounds on the number of $j$-dimensional strata. We ask whether similar results hold for Hurwitz slices defined by fixing the first $k$ coefficients. We note, however, that any stratification for weakly Hurwitz polynomials must necessarily be more refined than in the hyperbolic case: one has to distinguish between roots lying in the interior of the left half-plane, roots at the origin, and roots located on the remaining part of the boundary of the left half-plane. This suggests that the root multiplicity data should be described by triples $(s, r, \mu)$, where $s$ denotes the number of roots in the interior, $r$ denotes the multiplicity of the root at zero, and $\mu$ encodes the multiplicities of the roots on the rest of the boundary.

\subsubsection*{Acknowledgements} The authors would like to thank Thorsten Theobald for helpful insights on the semialgebraic structure of Hurwitz polynomials. We are also grateful to two anonymous referees for their constructive feedback, which helped to improve the presentation of the manuscript. 

\bibliographystyle{abbrv}
\bibliography{references}

\end{document}